\documentclass[reqno]{amsart}

\usepackage{amsmath}
\usepackage{amssymb}
\usepackage{mathrsfs} 

\theoremstyle{plain}
    \newtheorem{thm}{Theorem}[section]
    \newtheorem{ppn}[thm]{Proposition}
    \newtheorem{lem}[thm]{Lemma}
    \newtheorem{cor}[thm]{Corollary}
\theoremstyle{definition}
    
    \newtheorem{rmk}[thm]{Remark}
    \newtheorem{ex}[thm]{Example}

\numberwithin{equation}{section}

\allowdisplaybreaks

\def\Ker{\operatorname{Ker}}\def\Hom{\operatorname{Hom}}
\def\C{\mathbb{C}}\def\Q{\mathbb{Q}}\def\R{\mathbb{R}}\def\Z{\mathbb{Z}}\def\F{\mathbb{F}}
\def\ol#1{\overline{#1}}\def\wt#1{\widetilde{#1}}

\def\ot{\otimes}\def\ra{\rightarrow}
\def\a{\alpha}\def\b{\beta}\def\g{\gamma}
\def\l{\lambda}\def\k{\kappa}\def\z{\zeta}\def\vphi{\varphi}
\def\o{\omega}
\def\bbf{\mathbf{f}}
\def\sO{\mathscr{O}}\def\sD{\mathscr{D}}\def\sM{\mathscr{M}}
\def\angle#1{{\langle #1 \rangle}}
\def\G{\Gamma}
\def\L{\Lambda}\def\s{\sigma}\def\t{\tau}
\def\Im{\operatorname{Im}}
\def\prim{\mathrm{prim}}
\def\Mor{\operatorname{Mor}}
\def\ord{\operatorname{ord}}
\def\arg{\operatorname{arg}}
\def\Re{\operatorname{Re}}

\begin{document}

\title{On special values of Jacobi-sum Hecke $L$-functions}
\author{Noriyuki Otsubo}
\email{otsubo@math.s.chiba-u.ac.jp}
\address{Department of Mathematics and Informatics, Chiba University, Inage, Chiba, 263-8522 Japan}

\begin{abstract}
For motives associated with Fermat curves, there are elements in motivic cohomology whose regulators are written in terms of special values of generalized hypergeometric functions. 
Using them, we verify the Beilinson conjecture numerically for some cases and find formulae for 
the values of $L$-functions at $0$. 
These appear analogous to the Chowla-Selberg formula for the periods of elliptic curves with complex multiplication, which are related with the $L$-values at $1$ by the Birch and Swinnerton-Dyer conjecture.   
\end{abstract}

\date{\today}
\thanks{The author is supported by Researcher Exchange Program between JSPS and NSERC}

\maketitle

\section{Introduction}

Let $M$ be a motive over $\Q$ and suppose that the Hasse-Weil conjecture holds for $M$, 
i.e. its $L$-function $L(M,s)$ is analytically continued to $\C$ and satisfies a functional equation with respect to 
$s \leftrightarrow w+1-s$ where $w$ is the weight of $M$. 
Then, of great arithmetic interest is its behavior at integers,  
in particular the special value at $n \in \Z$, i.e. the first non-vanishing Taylor coefficient  
$$L^*(M,n) := \lim_{s \ra n} \frac{L(M,s)}{(s-n)^{\ord_{s=n}L(M,s)}}.$$ 
A sequence of important conjectures on this subject starts with the conjecture of Birch and Swinnerton-Dyer \cite{bsd}, extended by Tate \cite{tate}, for 
$h^1$ of an abelian variety $A$ and the central value $n=1$. 
In particular, if $L(h^1(A),1)\neq 0$, this value is conjectured to coincide with the real {\em period} of $A$ up to non-zero rational numbers. 

For motives with complex multiplication (CM), the periods take special forms. 
Already in 1897, Lerch \cite{lerch}, p.303, notices that the period of a CM elliptic curve is written as a product of 
values of the gamma function at rational numbers, which is now better known as the Chowla-Selberg formula \cite{c-s}. 
Typical examples of motives with CM by {\em abelian} fields are factors of the Jacobian varieties of Fermat curves. 
Conversely, Gross \cite{g-r-2} reduced the Chowla-Selberg formula (up to algebraic numbers) to the computation 
of periods of Fermat curves due to Weil \cite{weil-period} and Rohrlich \cite{g-r-2}. 
Let $X_N$ be the projective Fermat curve over $\Q$ whose affine equation is given by
$x^N+y^N=1$. 
Then, its complex periods are essentially special values of the beta function
$$B(\a,\b)=\frac{\G(\a)\G(\b)}{\G(\a+\b)}= \frac{\sin((\a+\b)\pi)}{\pi}\G(\a)\G(\b)\G(1-\a-\b)$$
with $\a, \b \in \frac{1}{N}\Z$. 

On the other hand, by Weil, \cite{weil-jacobi}, the $L$-function of $h^1(X_N)$ decomposes into those of the Jacobi-sum Hecke characters 
$j_N^{a,b}$ of the $N$-cyclotomic field $\Q(\z_N)$ indexed by $a, b \in \Z/N$ with $a, b, a+b \neq 0$ (see \S2.2). Accordingly, there exists a submotive $X_N^{[a,b]}$ of $h^1(X_N)$ such that 
$L(X_N^{[a,b]},s)=L(j_N^{a,b},s)$ (see \S2.1). 
Then, if it does not vanish at $s=1$, the BSD conjecture for $X_N^{[a,b]}$ reads
$$L(j_N^{a,b},1) \equiv \prod_h  \mathrm{Re}\bigl((1-\z_N^{ha})(1-\z_N^{hb})\bigr) B\left(\frac{\angle{ha}}{N},\frac{\angle{hb}}{N}\right) \pmod{\Q^\times}$$
where $\angle a \in \{1,2,\dots, N-1\}$ denotes the representative of $a$, 
and $h$ runs through elements of $(\Z/N)^\times$ such that $\angle{ha}+\angle{hb}<N$.
When $L(j_N^{a,b},1)\neq 0$, this congruence is known by the works of 
Damerell \cite{damerell}, Shimura \cite{shimura}, Blasius \cite{blasius} and Anderson \cite{anderson-taniyama}. 
Such examples of abelian CM motives played an important role when Deligne \cite{deligne} proposed the period conjecture for {\em critical} $L$-values. 

After the work of Bloch \cite{bloch-book} on CM elliptic curves, 
Beilinson \cite{beilinson} proposed a very general conjecture which relates $L^*(M,n)$ with the {\em regulator}. 
For the Dedekind $\z$-functions of number fields, the conjecture reduces to Dirichlet's class number formula when  $n=0$, and to Borel's theorem \cite{borel} when $n<0$ (formerly conjectured by Lichtenbaum \cite{lichtenbaum}). 
For a smooth projective curve $X$ over $\Q$, we have the regulator map 
$$r_\sD \colon H^{2}_\sM(X,\Q(2))_\Z \ra H^{2}_\sD(X_\R,\R(2)). $$
Here, the source is the integral part of the motivic cohomology with $\Q$-coefficients and the target is the real Deligne cohomology with $\R$-coefficients (see \S3.1 for concrete descriptions). 
The conjecture states firstly that $r_\sD \ot_\Q \R$ is an isomorphism, which implies that 
$\dim_\Q H_\sM^2(X,\Q(2))_\Z =g,$ the genus of $X$.
Secondly, it states that 
$$L^*(h^1(X),0) \equiv \det(r_\sD) \pmod{\Q^\times},$$
the determinant taken with respect to a canonical $\Q$-structure of the Deligne cohomology. 

In this paper, we verify numerically the conjecture, except for the injectivity of $r_\sD$, for some cases of Fermat motives $X_N^{[a,b]}$. 
We shall compute the regulators of certain special elements and the $L$-values independently and compare them. 
There are similar studies for elliptic curves by Bloch and Grayson \cite{bloch-grayson}, 
for a quotient of the Fermat quintic curve by Kimura \cite{k-kimura} (see Remark \ref{kimura}) and 
for families of hyperelliptic curves by Dokchitser, de Jeu and Zagier \cite{d-j-z}. 

For the Fermat curve, Ross \cite{ross} considered an element 
$e_N \in H_\sM^2(X_N,\Q(2))_\Z$ represented by the Milnor $K_2$-symbol $\{1-x,1-y\}$, 
and showed that $r_\sD(e_N) \neq 0$ for $N \geq 3$. 
In an earlier paper \cite{otsubo-1}, the author studied its projections $e_N^{[a,b]} \in H^2_\sM(X_N^{[a,b]},\Q(2))_\Z$ and expressed $r_\sD(e_N^{[a,b]})$ in terms of special values $F_N^{a,b}$ of generalized hypergeometric functions defined as follows (see Theorem \ref{regulator}). 
Recall that the hypergeometric function ${}_pF_q$ is defined by
$${}_pF_q\left({\a_1,\dots, \a_p \atop \b_1, \dots, \b_q};x\right)
=\sum_{n=0}^\infty \frac{\prod_{i=1}^p (\a_i,n)}{\prod_{j=1}^q(\b_j,n)}\frac{x^n}{n!}, $$
where $(\a,n)=\prod_{i=1}^n(\a+i-1)$, and it converges at $x=1$ if $\Re(\sum_j \b_j-\sum_i \a_i)>0$. 
Put for $\a, \b \not\in \Z$
\begin{equation*}
\wt{F}(\a,\b) = B(\a,\b)^2\; {}_3F_2\left({\a, \b, \a+\b-1 \atop \a+\b, \a+\b}; 1\right). 
\end{equation*}
We remark that $\wt F(\a,\b)$ can also be written using a special value of Appell's two-variable hypergeometric function $F_3$ (see loc. cit.).  
For $a, b \in \Z/N$ as before, we put 
\begin{equation*}\label{F_N^{a,b}}
F_N^{a,b}=\wt{F}\left(\frac{\angle a}{N},\frac{\angle b}{N}\right)-\wt{F}\left(\frac{\angle{-a}}{N},\frac{\angle{-b}}{N}\right).
\end{equation*}
Then, we have $F_N^{a,b} \neq 0$. 

If $(N, a, b)=1$, the conjectural dimension of the motivic cohomology is 
$g:=\phi(N)/2$ 
where $\phi$ denotes the Euler function. 
If $g=1$, our motive is isomorphic to $h^1$ of a CM elliptic curve and
the above result reproves the surjectivity of $r_\sD \ot_\Q\R$ due to Bloch.  
In \cite{otsubo-2}, we compared our elements with Bloch's elements and obtained the equalities 
\begin{equation*}
L^*(j_3^{1,1},0) = \frac{1}{6\sqrt{3}\pi} F_3^{1,1}, \quad
L^*(j_4^{1,2},0)= \frac{1}{8\pi}F_4^{1,2}.
\end{equation*}
Compare these with 
\begin{equation*}
L(j_3^{1,1},1) = \frac{1}{9}B\left(\frac{1}{3},\frac{1}{3}\right), \quad
L(j_4^{1,2},1)= \frac{1}{8}B\left(\frac{1}{4},\frac{2}{4}\right). 
\end{equation*}
The beta function itself is related with the special value 
of the Gauss hypergeometric function; in Euler's formula
$${}_2F_1\left({\a,\b \atop \g};1\right) = \frac{\G(\g)\G(\g-\a-\b)}{\G(\g-\a)\G(\g-\b)} \quad(\Re(\g-\a-\b)>0),$$
let $\g=\a+\b+1$ and use functional equations of the gamma function. 
Therefore, our formulae for regulators can be regarded as analogues of the Chowla-Selberg formula. 

When $g>1$, we use the action of the symmetric group of degree three on $X_N$ as the permutations of homogeneous coordinates to obtain more elements in the motivic cohomology.  
Using them, in \cite{otsubo-1}, the surjectivity of the regulator map is proved for some cases where $N$ is odd and $g \leq 3$. 
In \S3.4, we shall extend this method to even $N$ by once extending the base field to $\Q(\z_{2N})$. 

Then, in \S4, we compute numerically the regulator determinant, denoted by $R_N^{a,b}$, of our $g$ elements, which is a homogeneous polynomial in $F_N^{a,b}$'s. 
In all the plausible cases, we find that $R_N^{a,b} \neq 0$. 
Hence we obtain {\em rigorously} new cases where $r_\sD \ot_\Q\R$ is surjective. 
For the other cases, we have clear reasons why our elements are not sufficient.   
On the other hand, we compute numerically the values $L^*(j_N^{a,b},0)$ with the aid of Magma. 
In every case, we find that the ratio 
$L^*(j_N^{a,b},0)/R_N^{a,b}$ is close to a rational number as predicted by the Beilinson conjecture. 
Because of our expression of the regulators in terms of values of hypergeometric functions which converge rapidly, we can work with a high precision, at least $100$ digits. 
As a result, we shall find formulae such as
$$L^*(j_7^{1,2},0) \approx \frac{1}{49} R_7^{1,2}$$
where $R_7^{1,2}$ is a cubic polynomial in $F_7^{1,2}$, $F_7^{2,4}$ and $F_7^{4,1}$ (see \S4.3). 
As well as giving actual proofs of these formulae, it would be very tempting to find general formulae which express $L^*(j_N^{a,b}, n)$ for $n \leq 0$ in terms of special values of hypergeometric functions ${}_pF_q$, or hypergeometric functions of several variables. 

Our rational numbers turn out to be quite simple, only involving factors of $2N$. 
These are the subject of the Tamagawa number conjecture of Bloch and Kato \cite{b-k}. 
The computation of the $p$-adic regulators of our elements remains for a future study (see Remark \ref{p-reg}).  
On the other hand, one may hope to formulate an integral version of the Beilinson conjecture which 
generalizes the class number formula and the Lichtenbaum conjecture.  
Such a version would describe the $L$-value including the rational factor only in terms of the Beilinson regulator 
from the motivic cohomology with {\em integral} coefficients, together with information of other motivic cohomology groups. 
In the final \S5, we renormalize the results in \S4 in this framework, 
hoping that they might provide useful data for a future study. 

This paper is constructed as follows. 
In \S2, we recall basic facts about Fermat motives and Jacobi-sum Hecke $L$-functions while we fix notations. 
In \S3, we recall and extend the results of \cite{otsubo-1} on the regulator of Fermat motives. 
In \S4, we compute the regulators and the $L$-values, and compare them to find rational ratios.   
Finally in \S5, we work with integral coefficients and renormalize the results of \S4. 

\section{$L$-functions of Fermat motives}

Here we recall briefly Fermat motives over $\Q$, Jacobi sums and their $L$-functions. See \cite{otsubo-1} for the details. 

\subsection{Fermat motives}

Throughout this paper, we fix once and for all an embedding $\ol \Q \hookrightarrow \C$. 
For a positive integer $N$, put $\z_N = e^\frac{2 \pi i}{N}$, $K_N=\Q(\z_N)$, and let $\sO_N$ be the integer ring of $K_N$. For $h \in (\Z/N)^\times$, let $\s_h$ be the element of the Galois group 
$G(K_N/\Q)$ such that $\s_h(\z_N)=\z_N^h$.   

Let $X_N$ be the projective Fermat curve over $\Q$ whose affine equation is given by
$$x^N+y^N=1.$$
Define an index set
$$I_N=\left\{(a,b) \in (\Z/N)^{\oplus 2} \mid a, b, a+b \neq 0\right\}.$$ 
Then, $I_N$ is stable under the multiplication by $(\Z/N)^\times$. 
In the category of pure motives over $\Q$ with coefficients in $\Q$, we have a decomposition (\cite{otsubo-1}, \S2)
$$h^1(X_N)= \bigoplus_{[a,b] \in I_N/(\Z/N)^\times} X_N^{[a,b]}. $$
Here, the {\em Fermat motive} $X_N^{[a,b]}$ is defined as a pair $(X_N, p_N^{[a,b]})$ where $p_N^{[a.b]}$ is an algebraic correspondence on $X_N$, i.e. an element of the Picard group $\operatorname{Pic}(X_N \times X_N) \ot_\Z\Q$, which is idempotent with respect to the composition. 
Put 
$$G_N=(\Z/N)^{\oplus 2}$$ 
and write its element $(r,s)$ by $g^{r,s}$ and the addition multiplicatively: $g^{r,s}g^{r',s'}=g^{r+r',s+s'}$. 
Let $G_N$ act on $X_N \ot_\Q K_N$ by 
$$g^{r,s}(x,y)=(\z_N^r x,\z_N^s y).$$ 
For each $(a,b) \in I_N$, define a character $\theta_N^{a,b}$ of $G_N$ by
$$\theta_N^{a,b}(g^{r,s})= \z_N^{ar+bs}.$$ 
Define a correspondence on $X_N \ot_\Q K_N$ with $K_N$-coefficients by
$$p_N^{a,b} = \frac{1}{N^2} \sum_{g \in G_N} \theta^{a,b}(g)^{-1} \G_g, $$
where $\G_g$ denotes the transpose of the graph of $g$. Then, the orbit
$$p_N^{[a,b]}:=\sum_{(a',b') \in [a,b]} p_N^{a',b'}$$
is defined over $\Q$ and defines a correspondence on $X_N$ with $\Q$-coefficients. 
Since $p_N^{a,b}$ are idempotent, $p_N^{[a,b]}$ is also idempotent. 

On the level of realizations, for any cohomology theory $H^\bullet$ with coefficients in a field of characteristic zero, 
we define $H^\bullet(X_N^{[a,b]})$ to be the image of 
$$p_N^{[a,b] *} \colon H^\bullet(X_N) \ra H^\bullet(X_N).$$
Note that $\Gamma_g^*$ coincides with the pull-back $g^*$. 
For example, we write the Betti cohomology as
$$H^1(X^{[a,b]}_N(\C),\C) := p_N^{[a,b]*}H^1(X_N(\C),\C),$$ 
although there is not a space $X_N^{[a,b]}(\C)$.  
It has a basis $\{\o_N^{a',b'} \mid (a',b') \in [a,b]\}$, where $\o_N^{a,b}$ is the class of a $1$-form of the second kind defined by
$$\o_N^{a,b}=x^{\angle{a}}y^{\angle{b}-N}\frac{dx}{x}.$$
Recall that $\angle{a} \in \{1,2,\dots, N-1\}$ denotes the representative of $a\in \Z/N$. 
Note that these are eigenforms with respect to the $G_N$-action: 
$$g^*\o_N^{a,b}=\theta_N^{a,b}(g) \o_N^{a,b}.$$ 

By the morphism exchanging $x$ and $y$, we have an evident symmetry 
\begin{equation}\label{X1}
X_N^{[a,b]} \simeq X_N^{[b,a]}.
\end{equation}
If $N$ is odd, the symmetric group of degree $3$ acts on $X_N$ (see \S3.3), and we have also
\begin{equation}\label{X2}
X_N^{[a,b]}\simeq X_N^{[c,b]} \simeq X_N^{[a,c]}
\end{equation}
where, throughout this paper, $c \in \Z/N$ denotes the element defined by  
$$a+b+c=0.$$ 
We say that $(a,b)$ is {\em primitive} if $(N,a,b)=1$ and denote by $I_N^\prim\subset I_N$ the subset of primitive elements.  
If $(a,b)$ is not primitive, then there exists a unique $(a',b') \in I_{N/d}$ where $d=(N,a,b)$ such that 
$(a,b)=(da',db')$ and we have 
\begin{equation}\label{X3}
X_N^{[a,b]} \simeq X_{N/d}^{[a',b']}.
\end{equation}
Using these relations, we are reduced to study a smaller number of cases. 
The cases $g=1$, $2$, $3$, where $g:=\phi(N)/2$, 
are summarized as follows.  

\begin{ppn}\label{decomposition}
 In the category of motives over $\Q$, we have the following isomorphisms: 
\begin{align*}
& h^1(X_3) \simeq X_3^{[1,1]}, \\
& h^1(X_4) \simeq X_4^{[1,1]} \oplus (X_4^{[1,2]})^{\oplus 2}, \\
& h^1(X_5)\simeq (X_5^{[1,1]})^{\oplus 3}, \\
& h^1(X_6) \simeq h^1(X_3) \oplus X_6^{[1,1]} \oplus \bigoplus_{b=2,3,4} (X_6^{[1,b]})^{\oplus 2}
\oplus (X_6^{[2,3]})^{\oplus 2}, \\
& h^1(X_7) \simeq (X_7^{[1,1]})^{\oplus 3} \oplus (X_7^{[1,2]})^{\oplus 2}, \\
& h^1(X_8) \simeq h^1(X_4) \oplus \bigoplus_{b=1,3,5} X_8^{[1,b]} \oplus \bigoplus_{b=2,4,6} (X_8^{[1,b]})^{\oplus 2},\\
& h^1(X_9) \simeq h^1(X_3) \oplus (X_9^{[1,1]})^{\oplus 3} \oplus  (X_{9}^{[1,2]})^{\oplus 6},\\
& h^1(X_{10}) \simeq h^1(X_5) \oplus X_{10}^{[1,1]} \oplus \bigoplus_{b=2,3,4,5,6,8} (X_{10}^{[1,b]})^{\oplus 2} \oplus (X_{10}^{[2,5]})^{\oplus 2},\\
& h^1(X_{12}) \simeq h^1(X_4) \oplus h^1(X_6) 
\oplus \bigoplus_{b=1,4,5,7} X_{12}^{[1,b]} 
\oplus (X_{12}^{[1,2]})^{\oplus 3} 
\\& \phantom{AAAAAAA}
\oplus \bigoplus_{b=3,6,8,9,10} (X_{12}^{[1,b]})^{\oplus 2} 
\oplus (X_{12}^{[2,3]})^{\oplus 2}\oplus (X_{12}^{[3,4]})^{\oplus 2}, \\
& h^1(X_{14})\simeq h^1(X_7) \oplus X_{14}^{[1,1]} \oplus \bigoplus_{2 \leq b \leq 12, b \neq 5, 11} (X_{14}^{[1,b]})^{\oplus 2} \oplus 
(X_{14}^{[2,7]})^{\oplus 2},\\
& h^1(X_{18}) \simeq 
h^1(X_6) \oplus 
(X_9^{[1,1]})^{\oplus 3} \oplus  (X_{9}^{[1,2]})^{\oplus 6} \\
& \phantom{AAAAAAA}
\oplus X_{18}^{[1,1]} \oplus \bigoplus_{2 \leq b \leq 16, b \neq 11, 13} (X_{18}^{[1,b]})^{\oplus 2} 
\oplus \bigoplus_{b=3,9,15} (X_{18}^{[2,b]})^{\oplus 2}. 
\end{align*}
\end{ppn}

\begin{rmk}\label{C}
If $N$ is a prime, $X_N^{[a,b]}$ is isomorphic to $h^1(C_N^{a,b})$ of the curve
$$C_N^{a,b}\colon v^N=u^a(1-u)^b$$
via the morphism $u=x^N$, $v=x^ay^b$. 
\end{rmk}

\subsection{$L$-functions}

For a prime $v \nmid N$ of $K_N$, let $\F_v$ denote the residue field at $v$ and let $\chi_v \colon \F_v^\times \ra \mu_N$ be the $N$th power residue character, i.e. $\chi_v(x) \equiv x^{\frac{|\F_v|-1}{N}} \pmod v$. 
Then, for $(a, b) \in I_N$, the {\em Jacobi sum} is defined by
$$j_N^{a,b}(v) = -\sum_{x,y \in \F_v^\times, x+y=1} \chi_v^a(x)\chi_v^b(y).$$ 
The $L$-function of a Fermat motive, defined via its $\ell$-adic realization, coincides with that of the corresponding Jacobi sums (see \cite{otsubo-1}, Theorem 3.9): for $(a,b) \in I_N^\prim$, we have
$$L(X_N^{[a,b]},s) = L(j_N^{a,b},s):= \prod_{v \nmid N} \left(1- \frac{j_N^{a,b}(v)}{|\F_v|^s}\right)^{-1}.  
$$
Since $|j_N^{a,b}(v)|=|\F_v|^{1/2}$, it converges absolutely for $\Re(s)>3/2$.
We remark that $L(j_N^{a,b},s)$ depends only on the class $[a,b]$. In particular, 
$$L(j_N^{a,b},s)=L(j_N^{-a,-b},s)=L(\ol{j_N^{a,b}},s).$$
The symmetry \eqref{X1} (resp. \eqref{X2} for odd $N$) corresponds to the symmetry $j_N^{a,b}=j_N^{b,a}$ 
(resp. $j_N^{a,b}=j_N^{c,b}=j_N^{a,c}$ for odd $N$). 
For even $N$, we still have the following relation which reduces our computations. 
\begin{ppn}\label{L symmetry}
 Let $N$ be even and $(a, b) \in I_N$. If $b$ is even (resp. if $a$ is even), then we have
$j_N^{a,b}=j_N^{c,b}$ (resp. $j_N^{a,b}=j_N^{a,c}$). 
\end{ppn}

\begin{proof}This follows easily from
$j_N^{a,b}(v)=\chi_v(-1)^b j_N^{c,b}(v)=\chi_v(-1)^a j_N^{a,c}(v)$. 
\end{proof}

By Weil \cite{weil-jacobi}, the map $v \mapsto j_N^{a,b}(v)$, extended by linearity to the group of fractional ideals of $K_N$  prime to $N$, defines a Hecke character of a conductor dividing $N^2$.  
More precisely, put
$$H_N^{a,b}=\{h \in (\Z/N)^\times \mid \angle{ha} + \angle{hb} <N\},$$
and let
$$\sum_{h\in H_N^{a,b}} \s_h^{-1} \in \Z[G(K_N/\Q)]$$
be the {\em Stickelberger element}. 
Note that $h \in H_N^{a,b}$ if and only if $-h \not\in H_N^{a,b}$, i.e. the Stickelberger element is a CM type.  
Then, there exists an ideal $\bbf_N^{a,b}$ of $\sO_N$ dividing $N^2$ and a primitive finite character 
$$\vphi_N^{a,b} \colon (\sO_N/\bbf_N^{a,b})^\times \ra \C^\times$$
such that, for any $\a \in K^\times$ prime to $\bbf_N^{a,b}$, we have
$$j_N^{a,b}((\a)) = \vphi_N^{a,b}(\a)\prod_{h \in H_N^{a,b}}  \s_h^{-1}(\a). $$ 
In particular, for any unit $\a \in \sO_N^\times$, we have
$$\vphi_N^{a,b}(\a)^{-1} = \prod_{h \in H_N^{a,b}}  \s_h^{-1}(\a). $$ 

Therefore,  by the Hecke theory, $L(j_N^{a,b},s)$ is analytically continued to an entire function and satisfies the following functional equation. 
Let $d_N$ be the absolute value of the absolute discriminant of $K_N$, $\bbf_N^{a,b}$ be the conductor of $j_N^{a,b}$ as above  and $N(\bbf_N^{a,b})$ be its ideal norm. If we put
$$\L(j_N^{a,b}, s)= \left(d_N N(\mathbf{f}_N^{a,b})\right)^{s/2} \left(\frac{\G(s)}{(2 \pi)^s}\right)^g L(j_N^{a,b},s), $$
then we have 
$$\L(j_N^{a,b}, s)=\pm \L(j_N^{a,b}, 2-s).$$
It follows that $L(j_N^{a,b},s)$ has a zero of order $g$ at every nonpositive integer, and we consider the special value at $0$
\begin{equation}\label{f-e}
L^*(j_N^{a,b},0) := \lim_{s \ra 0} \frac{L(j_N^{a,b},s)}{s^g} = \pm \frac{d_N N(\bbf_N^{a,b})}{(2 \pi)^{2g} }L(j_N^{a,b},2) \ \in \R^\times.\end{equation}

\begin{rmk}
When $N$ is a prime, $\bbf_N^{a,b}$ was determined by Hasse \cite{hasse}. 
After several works, Coleman and McCallum \cite{c-m} determined $\bbf_N^{a,b}$ in general except for the factors dividing $2$. 
See also \cite{g-r}, where the sign of the functional equation (``root number") is determined when $N$ is a prime.  
\end{rmk}

\section{Regulators of Fermat motives}

We recall and extend the results of \cite{otsubo-1} on the regulator of Fermat motives.  
Most results in loc. cit. are stated for Fermat motives $X_N^{a,b}=(X_N\ot_\Q K_N, p_N^{a,b})$ over $K_N$ with coefficients in $K_N$. 
Here, we reformulate them as results  for $X_N^{[a,b]}$. 
In \S3.4, we extend the action of the symmetric group to the case where $N$ is even. 

\subsection{Beilinson's conjecture for curves}

Let $X$ be a smooth projective curve over $\Q$ of genus $g$. 
We recall a concrete description of the Beilinson regulator map  \cite{beilinson}
$$r_\sD \colon H^2_\sM(X,\Q(2))_\Z \ra H^2_\sD(X_\R,\R(2))$$
mentioned in \S1. For more details, see for example \cite{schneider}, \cite{d-j-z}. 

The source of $r_\sD$ is the integral part of the {\em motivic cohomology group}, 
which is a $\Q$-vector space defined using algebraic $K$-theory. 
For a field $k$, the second {\em Milnor $K$-group} $K_2^M(k)$ is the free abelian group generated by symbols $\{f,g\}$ ($f, g \in k^\times$) divided by the subgroup generated by Steinberg relations $\{f,1-f\}$ ($f \in k^\times, f \neq 1$). 
Let $\Q(X)$ be the function field of $X$, $X_0$ be the set of closed points on $X$ and $\Q(x)$ be the residue field at $x \in X_0$. Then, the {\em tame symbols}
\begin{equation*}
T \colon K_2^M(\Q(X)) \ra \bigoplus_{x \in X_0} \Q(x)^\times
\end{equation*}
are defined by sending $\{f, g\}$ to 
$$(-1)^{\ord_x(f) \ord_x(g)} \left(\frac{f^{\ord_x(g)}}{g^{\ord_x(f)}}\right)(x). $$
Then, we have a natural isomorphism 
$$\Ker(T)\ot_\Z \Q \simeq H_\sM^2(X,\Q(2)).$$
The {\em integral part} $H_\sM^2(X,\Q(2))_\Z$ consists of 
those elements which can be extended to a regular model of $X$ which is proper and flat over $\Z$.   

The target of $r_\sD$ is the {\em real Deligne cohomology group} and in this case, we have a natural  isomorphism
$$H^2_\sD(X_\R,\R(2)) \simeq H^1(X(\C),\R(1))^+.$$ 
Here, for any subring $A \subset \R$, we write $A(1)=2\pi i A$ on which the complex conjugation $c_\infty$ acts by $-1$. 
On the other hand, we have the complex conjugation $F_\infty$ (infinite Frobenius) acting on $X(\C)$. 
The script ``$+$" denotes the part fixed by $F_\infty \ot c_\infty$. 
It is endowed with the canonical $\Q$-structure 
$$H^1(X(\C),\Q(1))^+ =2 \pi i \cdot H^1(X(\C),\Q)^{F_\infty=-1}.$$ 
Under these identifications, the regulator map $r_\sD$ sends $\sum_n \{f_n,g_n\} \in \Ker(T)$ to 
$$i \cdot \sum_n (\log |f_n| d \arg g_n-\log |g_n| d \arg f_n). $$

As explained in \S1, Beilinson's conjecture states firstly that, $r_\sD \ot_\Q \R$ is an isomorphism, and hence
$$\dim_\Q H_\sM^2(X,\Q(2))_\Z=g.$$ 
Assuming the first, the determinant $R$ of a matrix expressing $\Im(r_\sD)$ with respect to the $\Q$-structure of the Deligne cohomology is well-defined in $\R^\times/\Q^\times$, 
and the conjecture states secondly that 
$$L^*(h^1(X),0) \equiv R \pmod{\Q^\times},$$  
where the analytic continuation of $L(h^1(X),s)$ is presumed. 

The conjecture extends naturally to motives associated with curves by taking projections. 
When $g>0$, nothing is known about the finite generation of the motivic cohomology 
and the injectivity of $r_\sD$.  
As in other related studies, we only consider a weaker version of the conjecture that there exists a $g$-dimensional subspace of the motivic cohomology having the desired properties. 
In other words, we consider the conjecture admitting the injectivity of $r_\sD$. 
A version with $\Z$-structures instead of $\Q$-structures will be discussed in \S5. 

\subsection{Regulators}

Let $X_N$ be the Fermat curve over $\Q$ as before. 
In this case,  $H^2_\sM(X_N,\Q(2))_\Z=H^2_\sM(X_N,\Q(2))$. 
Since
the tame symbols of $\{1-x,1-y\} \in K_2(\Q(X_N))$ are $2N$-torsion (\cite{ross}, Theorem 1), 
it defines an element 
$$e_N := 2N\{1-x,1-y\} \in H^2_\sM(X_N,\Q(2))_\Z.$$
For $(a,b) \in I_N$, we put
$$e_N^{[a,b]}:=p_N^{[a,b]*}(e_N) \in H_\sM^2(X_N^{[a,b]},\Q(2))_\Z.$$ 

\begin{rmk}
We changed the notation from \cite{otsubo-1}; $e_N$ and $e_N^{[a,b]}$ here are $2 N$ times those of loc. cit., so that they belong to the $\Z$-structure in the sense of \S5. 
\end{rmk}

We describe the image of $e_N^{[a,b]}$ under the regulator map
\begin{equation*}
r_\sD \colon H_\sM^2(X_N^{[a,b]},\Q(2))_\Z \ra H_\sD^2(X_{N,\R}^{[a,b]},\R(2)). 
\end{equation*}
For $(a,b)\in I_N$, let $\o_N^{a,b} \in H^1(X_N(\C),\C)$ be as in \S2.1 and 
we normalize it as
$$\wt\o_N^{a,b} := \left(\frac{1}{N} B\left(\frac{\angle a}{N}, \frac{\angle b}{N}\right)\right)^{-1} 
\o_N^{a,b}.$$
Then, since we have
$$\int_\g F_\infty\wt\o_N^{a,b}=  \int_\g \wt\o_N^{-a,-b} = \ol{\int_\g \wt\o_N^{a,b}} \ \in K_N$$
for all $\g \in H_1(X_N(\C),\Z)$,  
$$\wt\o_N^{a,b} \in H^1(X_N(\C),K_N), \quad F_\infty \wt\o_N^{a,b} =\wt\o_N^{-a,-b}= c_\infty \wt\o_N^{a,b}.$$
Therefore, if $(a,b) \in I_N^\prim$, the set 
$$\left\{\wt\o_N^{ha,hb}-\wt\o_N^{-ha,-hb} \bigm| h \in H_N^{a,b}\right\}$$
gives a basis of the $K_N$-vector space $H^1(X^{[a,b]}_N(\C),K_N)^+$.  

Now, the results of \cite{otsubo-1} (Theorem 4.14, Proposition 4.25) can be read as follows. 
For $(a,b) \in I_N$, let $F_N^{a,b} \in \R$ be as defined in \S1. Note that 
$$F_N^{-a,-b}=-F_N^{a,b}.$$ 
By \eqref{X3}, we are reduced to consider the primitive cases. 

\begin{thm}\label{regulator} 
For $(a,b) \in I_N^\prim$,  
we have 
\begin{align*}
r_\sD(e_N^{[a,b]})
= - \frac{1}{N} \sum_{h \in (\Z/N)^\times} F_N^{ha,hb} \wt\o_N^{ha,hb}
= - \frac{1}{N} \sum_{h \in H_N^{a,b}} F_N^{ha,hb} (\wt\o_N^{ha,hb}-\wt\o_N^{-ha,-hb}). 
\end{align*}
Moreover, $F_N^{ha,hb} \neq 0$ for all $h \in (\Z/N)^\times$ and $F_N^{ha,hb}>0$ if and only if $h \in H_N^{a,b}$. 
\end{thm}

In particular, $e_N^{[a,b]}$ is non-trivial and if $g=1$, then $r_\sD\ot_\Q\R$ is surjective. 
We can transform the theorem into an expression of the regulator with respect to the $\Q$-structure 
using the following: 

\begin{ppn}\label{Q-structure}
For any $(a,b) \in I_N^\prim$, there exists a basis $\{\l_n \mid 1 \leq n\leq g\}$ of the $\Q$-vector space 
$H^1(X_N^{[a,b]}(\C),\Q(1))^+$ such that 
$$\wt\o_N^{ha,hb}-\wt\o_N^{-ha,-hb} = \frac{1}{2\pi i} \sum_{n=1}^g  (\z_N^{hn}-\z_N^{-hn}) \l_n$$
for all $h \in (\Z/N)^\times$. 
\end{ppn}

\begin{proof}
The argument is similar to \cite{otsubo-1}, Corollary 4.21. See also Remark \ref{kappa gamma}. 
\end{proof}

Define for the later use
\begin{equation}\label{D}
D_N:=\left|\det (\z_N^{hn}-\z_N^{-hn})_{h, n}\right|
=\left|\det\left(2 \sin \frac{2\pi hn}{N}\right)_{h, n}\right| 
\end{equation}
where $h$ runs through the $g$ elements of $(\Z/N)^\times$ with $\angle h <N/2$ and 
$n=1,2,\dots, g$. 
One calculates for example 
\begin{equation*}
\begin{split}
& D_3=D_6= \sqrt{3}, \quad  D_4=2,  \\
& D_5=D_{10}=5, \quad D_8=4\sqrt{2}, \quad D_{12}=2 \sqrt{3},\\
& D_7=D_{14}=7\sqrt{7}, \quad  D_9=D_{18}=9\sqrt{3}. 
\end{split}\end{equation*}
Note that $D_N$ is independent of $(a,b)$. 

\subsection{Symmetry for odd $N$}

Let $N$ be odd and define involutions $\a$, $\b$ on $X_N$ by 
$$\a(x,y)=\left(\frac{1}{x},-\frac{y}{x}\right), \quad \b(x,y)=\left(-\frac{x}{y},\frac{1}{y}\right). $$ 
In the projective equation
$$x_0^N+y_0^N=z_0^N,$$
$\a$ (resp. $\b$) exchanges $x_0$ and $-z_0$ (resp. $y_0$ and $-z_0$). 
We remark that exchanging $x$ and $y$ is useless, since it acts on $e_N$ by $-1$ because the 
Milnor symbol is skew-symmetric. 
The actions of $\a$, $\b$ on the Betti cohomology are given as follows (see \cite{otsubo-1}, Lemma 4.30). 

\begin{lem}\label{odd N}
For $(a,b) \in I_N$ with $\angle a + \angle b + \angle c=N$, we have
\begin{alignat*}{2}
& \a^* \wt\o_N^{a,b} =(-1)^{\angle b} \frac{\sin \frac{\angle{a}\pi}{N}}{\sin \frac{\angle{c}\pi}{N}}\wt\o_N^{c,b}, \quad 
&&  \a^* \wt\o_N^{-a,-b} =(-1)^{\angle b} \frac{\sin \frac{\angle{a}\pi}{N}}{\sin \frac{\angle{c}\pi}{N}}\wt\o_N^{-c,-b},\\
& \b^* \wt\o_N^{a,b} =(-1)^{\angle a} \frac{\sin \frac{\angle{b}\pi}{N}}{\sin \frac{\angle{c}\pi}{N}}\wt\o_N^{a,c}, 
&& \b^* \wt\o_N^{-a,-b} =(-1)^{\angle a} \frac{\sin \frac{\angle{b}\pi}{N}}{\sin \frac{\angle{c}\pi}{N}}\wt\o_N^{-a,-c}.
\end{alignat*}
\end{lem}

We define elements of $H^2_\sM(X_N^{[a,b]},\Q(2))_\Z$ by
\begin{equation*}
e_{N,\a}^{a,b} = p_N^{[a,b] *}\circ \a^* (e_N), \quad e_{N,\b}^{a,b} = p_N^{[a,b]*}\circ \b^* (e_N).
\end{equation*}
Since
$$p_N^{[a,b] *}\circ \a^* =\a^*\circ p_N^{[c,b] *}, \quad p_N^{[a,b] *} \circ \b^* =\b^*\circ p_N^{[a,c] *}, $$
and the regulator map is compatible with the actions of $p_N^{[a,b]}$, $\a$, $\b$, we obtain the following proposition 
from Theorem \ref{regulator}. Note that $H_N^{a,b}=H_N^{c,b}=H_N^{a,c}$. 

\begin{ppn}
Let $N$ be odd. For any $(a,b) \in I_N^\prim$, we have
\begin{align*}
& r_\sD(e_{N,\a}^{[a,b]}) = -\frac{1}{N} \sum_{h\in H_N^{a,b}} (-1)^{\angle{hb}} \frac{\sin\frac{\angle{hc}\pi}{N}}{\sin\frac{\angle{ha}\pi}{N}}F_N^{hc,hb} (\wt\o_N^{ha,hb}-\wt\o_N^{-ha,-hb}),\\
& r_\sD(e_{N,\b}^{[a,b]}) = -\frac{1}{N} \sum_{h\in H_N^{a,b}} (-1)^{\angle{ha}} \frac{\sin\frac{\angle{hc}\pi}{N}}{\sin\frac{\angle{hb}\pi}{N}}F_N^{ha,hc} (\wt\o_N^{ha,hb}-\wt\o_N^{-ha,-hb}). 
\end{align*}
\end{ppn}

In \cite{otsubo-1}, it is shown that the three elements $e_N^{[a,b]}$, $e_{N,\a}^{[a,b]}$, $e_{N,\b}^{[a,b]}$ are sufficient for the surjectivity of $r_\sD \ot_\Q \R$ for $X_5^{[1,1]}$ (hence for the whole $X_5$) 
and $X_7^{[1,2]}$, but not for $X_7^{[1,1]}$. 
In general, the possibility of liner independence is restricted by the following. 

\begin{cor}\label{restriction}
If $a=b$ (resp. $a=c$, $b=c$), then we have $r_\sD(e_{N,\a}^{[a,b]})=r_\sD(e_{N,\b}^{[a,b]})$ 
(resp. $r_\sD(e_{N,\a}^{[a,b]})=r_\sD(e_N^{[a,b]})$, $r_\sD(e_{N,\b}^{[a,b]})=r_\sD(e_N^{[a,b]})$). 
\end{cor}

\subsection{Symmetry for even $N$}
Let  $N$ be even.  Then, we do not have $\a$, $\b$ as above defined over $\Q$. 
For simplicity, we write
$$L=K_{2N}, \quad K=K_N, \quad X_L = X_N \ot_\Q L, \quad X_K=X_N \ot_\Q K,$$ 
and let 
$$f_L \colon X_L \ra X_N, \quad f_K \colon X_K \ra X_N, \quad f_{L/K} \colon X_L \ra X_K$$
be the natural morphisms. 
Define involutions $\a_L$, $\b_L$ on $X_L$ by
$$\a_L(x,y) = \left(\frac{\z_N}{x}, \frac{\z_{2N}y}{x}\right), \quad 
\b_L(x,y) = \left(\frac{\z_{2N}x}{y}, \frac{\z_{N}}{y}\right),$$
and elements of $H_\sM^2(X_N^{[a,b]},\Q(2))_\Z$ by
\begin{equation*}
\begin{split}
e_{N,\a}^{[a,b]}= p_N^{[a,b] *} \circ {f_L}_* \circ \a_L^* \circ f_L^*(e_N), \\
e_{N,\b}^{[a,b]}= p_N^{[a,b] *} \circ {f_L}_* \circ \b_L^* \circ f_L^*(e_N),
\end{split}\end{equation*}
where $f_{*}$ denotes the push-forward map for a proper morphism $f$. 

Let us see the action of $\a_L$ and $\b_L$ on the Betti cohomology of $X_L$. 
Since $X_L(\C) = \Mor_\Q(\C,X_L)$ is the disjoint union of $X_{L,\tau}(\C) = \Mor_{L,\tau}(\C,X_L)$ for all embeddings $\tau \colon L \hookrightarrow \C$, it suffices to describe the actions on each component. 
Let $\{\wt\o_\tau^{a,b} \mid (a,b) \in I_N\}$ be the basis of $H^1(X_{L,\tau}(\C),\C)$ defined as before. 

\begin{lem}
Let $\tau \colon L \hookrightarrow \C$ be an embedding and $\ol\t$ be its complex conjugate.
For $(a,b) \in I_N$ with $\angle a + \angle b + \angle c =N$, 
 we have
\begin{alignat*}{2}
& \a_L^*\wt\o_\tau^{a,b}= \t(\z_N^\angle{a} \z_{2N}^{\angle b}) \frac{\sin \frac{\angle{a}\pi}{N}}{\sin \frac{\angle{c}\pi}{N}} \wt\o_\t^{c,b}, \quad 
&& \a_L^*\wt\o_\tau^{-a,-b}= \ol\t(\z_N^\angle{a} \z_{2N}^{\angle b}) \frac{\sin \frac{\angle{a}\pi}{N}}{\sin \frac{\angle{c}\pi}{N}} \wt\o_\t^{-c,-b}, 
\\
& \b_L^*\wt\o_\tau^{a,b}= \t(\z_N^\angle{b} \z_{2N}^{\angle a}) \frac{\sin \frac{\angle{b}\pi}{N}}{\sin \frac{\angle{c}\pi}{N}} \wt\o_\t^{a,c}, 
&& \b_L^*\wt\o_\tau^{-a,-b}= \ol\t(\z_N^\angle{b} \z_{2N}^{\angle a}) \frac{\sin \frac{\angle{b}\pi}{N}}{\sin \frac{\angle{c}\pi}{N}} \wt\o_\t^{-a,-c}. 
\end{alignat*} 
\end{lem}

\begin{proof}
This is proven similarly as Lemma \ref{odd N}. 
The left equations are direct, from which the right ones are deduced using 
$c_\infty \wt\o_\t^{a,b}=\wt\o_\t^{-a,-b}$.
\end{proof}

Let $T_{K_N/\Q}\colon K_N \ra \Q$ denote the trace map. 
Since $N$ is even, we just say that $a \in \Z/N$ is odd or even depending on the parity of $\angle a$.  

\begin{ppn}\label{even N}
Let $N$ be even and $(a,b) \in I_N^\prim$. 
\begin{enumerate}
\item
If $b$ is odd, then $r_\sD(e_{N,\a}^{[a,b]}) =0$. If $b$ is even, we have
$$
r_\sD(e_{N,\a}^{[a,b]}) =- \frac{2}{N} \sum_{h\in H_N^{a,b}}  T_{K_N/\Q}(\z_N^{\angle{hc}+\angle{hb}/2}) \frac{\sin\frac{\angle{hc}\pi}{N}}{\sin\frac{\angle{ha}\pi}{N}}F_N^{hc,hb}(\wt\o_N^{ha,hb}-\wt\o_N^{-ha,-hb}).$$
\item 
If $a$ is odd, then $r_\sD(e_{N,\b}^{[a,b]}) =0$. If $a$ is even, we have 
$$
r_\sD(e_{N,\b}^{[a,b]}) =- \frac{2}{N} \sum_{h\in H_N^{a,b}}  T_{K_N/\Q}(\z_N^{\angle{hc}+\angle{ha}/2}) \frac{\sin\frac{\angle{hc}\pi}{N}}{\sin\frac{\angle{hb}\pi}{N}}F_N^{ha,hc}(\wt\o_N^{ha,hb}-\wt\o_N^{-ha,-hb}). 
$$
\end{enumerate}
\end{ppn}

\begin{proof}
We only prove (i) and the proof of (ii) is parallel. 
By \cite{otsubo-1}, Theorem 4.14, the $\t$-component 
$r_{\sD,\t}$ of the regulator for $X_L$ is given by
$$r_{\sD,\t}(f_L^*(e_N)) = -\frac{1}{N} \sum_{(a,b) \in I_N} F_N^{ha,hb} \wt\o_{N,\t}^{ha,hb},$$
where $h \in (\Z/N)^\times$ is the element satisfying $\t(\z_N)^h=\z_N$. 
Let $\t'$ be the conjugate of $\t$ over $K$ and $\t_0$ be the restriction of $\t$ to $K$. 
Then we have 
$$f_{L/K *}(\wt\o_\t^{a,b})=f_{L/K *}(\wt\o_{\t'}^{a,b})=\wt\o_{\t_0}^{a,b},$$
where  
$\wt\o_{\t_0}^{a,b} \in H^1(X_{K,\t_0}(\C),\C)$ is similarly defined. 
By the lemma above  
and the compatibility of the regulator maps with pull-backs and push-forwards, the coefficient of 
$\wt\o_{\t_0}^{a,b}$ in the expression of $r_{\sD,\t_0}(f_{L/K *}\circ \a^* \circ f_L^*(e_N))$ is 
\begin{align*}
\bigl(\t(\z_N^\angle{c}\z_{2N}^{\angle b}) + \t'(\z_N^c\z_{2N}^{\angle b})\bigr) F_N^{c,b}
& =\t_0(\z_N^\angle{c})\t(\z_{2N}^{\angle{b}})\bigl(1+(-1)^{\angle{b}}\bigr) F_N^{c,b}\\
&=\begin{cases}
0, & \text{if $b$ is odd}, \\
2 \t_0(\z_N^{\angle{c}+\angle{b}/2}) F_N^{a,b}, & \text{if $b$ is even}. 
\end{cases}
\end{align*}
This proves the case when $b$ is odd. Since $f_{K *}(\wt\o_{\t_0}^{a,b})=\wt\o_N^{a,b}$,  
$f_{L *}=f_{K *} \circ f_{L/K *}$, and $p_N^{[a,b] *}$ acts on $\wt\o_N^{a',b'}$ identically if $(a',b') \in [a,b]$ and trivially otherwise, we obtain the other case. 
\end{proof}

Since $N$ is even and $(a,b) \in I_N^\prim$, at least one of $a$, $b$ is odd. 
Therefore, if $g=3$, our elements are not sufficient for the surjectivity of the regulator map. 
Hence the cases $N=14$, $18$ will be excluded.  
Similarly as Corollary \ref{restriction}, we have the following.  

\begin{cor}\label{restriction even}
If $a=b$ (resp. $a=c$, $b=c$), then we have $r_\sD(e_{N,\a}^{[a,b]})=r_\sD(e_{N,\b}^{[a,b]})$ 
(resp. $r_\sD(e_{N,\a}^{[a,b]})= -2\phi(N) \cdot r_\sD(e_N^{[a,b]})$, $r_\sD(e_{N,\b}^{[a,b]})=-2\phi(N) \cdot r_\sD(e_N^{[a,b]})$). 
\end{cor}

\begin{proof}
The first case is obvious. In the second case, since $2\angle{hc}+\angle{hb}=N$ for any $h \in H_N^{a,b}$, 
we have $T_{K_N/\Q}(\z_N^{\angle{hc}+\angle{hb}/2})=T_{K_N/\Q}(-1)=-\phi(N)$. 
The third case is parallel. 
\end{proof}

\section{Comparisons}

In this section, we determine when our three elements $e_N^{[a,b]}$, $e_{N,\a}^{[a,b]}$, $e_{N,\b}^{[a,b]}$ are sufficient for the surjectivity of 
$$r_\sD \ot_\Q \R \colon H_\sD^2(X_N^{[a,b]},\Q(2))_\Z \ot_\Q \R \ra H_\sD^2(X_{N,\R}^{[a,b]},\R(2)).$$
For each $N$, we use Proposition \ref{decomposition} to reduce the number of $(a,b)$'s to study. 
In each case, we compute the regulator determinant $R_N^{a,b}$ 
and the $L$-value $L^*(j_N^{a,b},0)$ numerically, and compare them. 
Here, $R_N^{a,b}$ is computed with respect to the basis of the $\Q$-structure of Deligne cohomology given in Proposition \ref{Q-structure}. An integral refinement will be given in \S5.  

To compute the $L$-values, we used Magma (see \cite{magma}, Part V, 32.9.6). 
It gives directly $L^*(j_N^{a,b},0)$ when $g=1$, while 
we computed $L(j_N^{a,b},2)$ and used the functional equation \eqref{f-e} when $g=2$ and $3$. 
We used Mathematica for other computations. 

\subsection{$g=1$} 

For $N=3, 4, 6$, and $(a, b) \in I_N^\prim$, we only need $e_N^{[a,b]}$.  
By Theorem \ref{regulator} and Proposition \ref{Q-structure}, the regulator is given by
\begin{equation*}
R_N^{a,b} := \frac{D_N}{2 \pi N} |F_N^{a,b}|, 
\end{equation*}
where $D_N$ is as defined in \eqref{D}.  
Then the results of \cite{otsubo-2} recalled in \S1 are written as:
\begin{align}
& L^*(j_3^{1,1},0)=\frac{1}{3} R_3^{1,1}, \\
& L^*(j_4^{1,2},0)=\frac{1}{2} R_4^{1,2}.
\end{align}
We verify the remaining cases. 

For $N=4$, by Proposition \ref{decomposition}, it remains to study the case $(a,b)=(1,1)$. 
Then, one computes
$$R_4^{1,1}=3.3173289967638281780989923863189664030737625416964\dots 
$$
On the other hand, the conductor of $j_4^{1,1}$ is
$$\mathbf{f}_4^{1,1}=(4)=(1-\z_4)^4,$$ 
and the CM type is $\s_1$. 
The group $(\sO_4/\bbf_4^{1,1})^\times$ is generated by $\z_4 \in \sO_4^\times$ and $1-2\z_4$. 
Since $v=(1-2\z_4)$ is a prime ideal of degree one above $5$ such that $\chi_v(2)=\z_4^{-1}$,   
one easily computes that $j_4^{1,1}((1-2\z_4))=1-2\z_4$, hence we have 
$$\vphi_4^{1,1}(1-2\z_4)=1.$$ 
Using these data, one computes 
$$L^*(j_4^{1,1},0)=
1.6586644983819140890494961931594832015368812708482\dots 
$$ 
By comparison, we obtain:
\begin{equation}
L^*(j_4^{1,1},0) \approx \frac{1}{2} R_4^{1,1}
\end{equation}
with $100$ digits precision. 

For $N=6$, by Proposition \ref{decomposition}, we are reduced to 
see the cases $(a,b)=(1,1)$, $(1,2)$, $(1,3)$, $(1,4)$ and $(2,3)$. 
Moreover, by Proposition \ref{L symmetry}, we have
$$L(j_6^{1,2},s)=L(j_6^{2,3},s).$$ 
In any case, the conductor of $j_6^{a,b}$ divides $(12)=(2)^2(1-\z_3)^2$, and is as listed below. 
The CM type is $\s_1$ for any case. 
The group $(\sO_3/(12))^\times$ is generated by $-1 \in \sO_3^\times$, $s:=-2-3\z_3$ and $t:= 1+4\z_3$.  
Since $(s)$, $(t)$ are prime ideals of degree one above $7$, $13$, respectively, 
the Jacobi sums $j_6^{a,b}((s))$, $j_6^{a,b}((t))$ are easily computed. 
As a result, we have 
\begin{alignat*}{3}
& \bbf_6^{1,1}=(2)^2(1-\z_3)^2, \quad && \vphi_6^{1,1}(s)=-\z_3, \quad && \vphi_6^{1,1}(t)=-\z_3^2,\\
& \bbf_6^{1,2}=(2)(1-\z_3), && \vphi_6^{1,2}(s)=\z_3^2, \quad && \vphi_6^{1,2}(t)=-1,\\
& \bbf_6^{1,3}=(2)^2(1-\z_3), && \vphi_6^{1,3}(s)=-\z_3^2, \quad && \vphi_6^{1,3}(t)=-1,\\
& \bbf_6^{1,4}=(2)(1-\z_3)^2, && \vphi_6^{1,4}(s)=\z_3, \quad && \vphi_6^{1,4}(t)=-\z_3^2. \\
\end{alignat*}
By numerical computations, we obtain
\begin{align}
& L^*(j_6^{1,1},0) \approx -2R_6^{1,1},\\
& L^*(j_6^{1,2},0) \approx \frac{1}{3} R_6^{1,2},\\
& L^*(j_6^{1,3},0) \approx 2R_6^{1,3},\\
& L^*(j_6^{1,4},0) \approx 2R_6^{1,4}, \\
& L^*(j_6^{2,3},0) \approx 2R_6^{2,3}, 
\end{align}
with $100$ digits precision. 

\begin{rmk}\label{hg relation 1}
The above results suggest an equality 
$F_6^{1,2} = 6 F_6^{2,3}$, although $\wt F(\frac{1}{6},\frac{2}{6}) \neq 6 \wt F(\frac{2}{6}, \frac{3}{6})$. 
The author does not know if this follows from known relations among ${}_3F_2$-values.  
\end{rmk}

\subsection{$g=2$}

We consider the cases $N=5, 8, 10, 12$. 

For $N=5$, by Proposition \ref{decomposition}, we are reduced to the case $(a,b)=(1,1)$. 
By Theorem \ref{regulator}, Propositions \ref{Q-structure} and \ref{odd N}, the regulator with respect to $e_5^{[1,1]}$ and $e_{5,\a}^{[1,1]}$ is given by
\begin{equation*}
R_5^{1,1}:=\frac{D_5}{(2\pi )^2 5^2} \left| \det\begin{pmatrix}
F_5^{1,1} & -\frac{\sin\frac{3 \pi}{5}}{\sin\frac{\pi}{5}} F_5^{3,1} \\
F_5^{2,2} & \frac{\sin\frac{\pi}{5}}{\sin\frac{2 \pi}{5}} F_5^{1,2}.
\end{pmatrix}\right|.
\end{equation*}
One sees $R_5^{1,1} \neq 0$ without numerical computation and obtains the surjectivity of $r_\sD \ot_\Q\R$ 
(\cite{otsubo-1}, Theorem 4.33).
On the other hand, the conductor of $j_5^{1,1}$ is 
$$\mathbf{f}_5^{1,1}=(1-\z_5)^2,$$ 
and 
the CM type is $\s_1+\s_3$. 
The group $(\sO_5/\bbf_5^{1,1})^\times$ is generated by $\z_5 \in \sO_5^\times$ and $2$.  
Since $2$ is inert in $K_5$, we have $j_5^{1,1}((2))=-2^2$ and hence 
$$\vphi_5^{1,1}(2)=-1.$$ 
By numerical computations, we obtain
\begin{equation}
L^*(j_5^{1,1},0) \approx \frac{1}{5} R_5^{1,1}
\end{equation}
with $100$ digits precision. 

\begin{rmk}\label{kimura}
This improves a result of Kimura \cite{k-kimura} on the curve $C_5^{1,1}$ (see Remark \ref{C}). 
For the relation with his elements in the motivic cohomology with ours, see \cite{otsubo-1}, Remark 4.35. 
\end{rmk}

For $N=8$, by Proposition \ref{decomposition}, we are reduced to the cases where $a=1$. 
By Proposition \ref{even N}, we can only treat the case when $b$ is even, 
by using $e_8^{[a,b]}$ and $e_{8,\a}^{[a,b]}$.  
However, one sees from Proposition \ref{even N} and Corollary \ref{restriction even} respectively that
$$r_\sD(e_{8,\a}^{[1,2]})=r_\sD(e_{8,\a}^{[1,4]})=0, \quad r_\sD(e_{8,\a}^{[1,6]})=-8 \cdot r_\sD(e_{8}^{[1,6]}).$$ 
Therefore, in any case, we have only one linearly independent element. 

For $N=10$, by Proposition \ref{decomposition}, we are reduced to the cases where $a=1$ and $(a,b)=(2,5)$. 
When $a=1$, we can possibly use $e_{10}^{[a,b]}$, $e_{10,\a}^{[a,b]}$ for even $b$. 
The case $b=8$, however, is excluded by Corollary \ref{restriction even}. 
For $(a,b)=(2,5)$, we use $e_{10}^{[2,5]}$ and $e_{10,\b}^{[2,5]}$.  
By Theorem \ref{regulator}, Propositions \ref{Q-structure} and \ref{even N}, the regulator determinant 
with respect to these elements are: 

\begin{align*}
& R_{10}^{1,2} := \frac{D_{10}}{(2\pi)^2 10^2}  \left| \det
\begin{pmatrix}
F_{10}^{1,2} & - 2\frac{\sin\frac{7\pi}{10}}{\sin\frac{\pi}{10}} F_{10}^{7,2}\\
F_{10}^{3,6} & - 2\frac{\sin\frac{\pi}{10}}{\sin\frac{3\pi}{10}} F_{10}^{1,6}
\end{pmatrix} \right|, \\
& R_{10}^{1,4} :=\frac{D_{10}}{(2\pi)^2 10^2}   \left| \det
\begin{pmatrix}
F_{10}^{1,4} & \frac{2}{\sin\frac{\pi}{10}} F_{10}^{5,4}\\
F_{10}^{3,2} &  - \frac{2}{\sin\frac{3\pi}{10}} F_{10}^{5,2}
\end{pmatrix} \right|, \\
& R_{10}^{1,6} := \frac{D_{10}}{(2\pi)^2 10^2}   \left| \det
\begin{pmatrix}
F_{10}^{1,6} & - 2\frac{\sin\frac{3\pi}{10}}{\sin\frac{\pi}{10}} F_{10}^{3,6}\\
F_{10}^{7,2} & - 2\frac{\sin\frac{\pi}{10}}{\sin\frac{7\pi}{10}} F_{10}^{1,2}
\end{pmatrix} \right|, \\
& R_{10}^{2,5} := \frac{D_{10}}{(2\pi)^2 10^2}   \left| \det
\begin{pmatrix}
F_{10}^{2,5} & - 2\sin\frac{3\pi}{10} \cdot F_{10}^{2,3}\\
F_{10}^{4,5} &  2 \sin\frac{\pi}{10} \cdot F_{10}^{4,1}
\end{pmatrix} \right|. 
\end{align*}
For the $L$-functions, we have by Proposition \ref{L symmetry}
$$L(j_{10}^{2,5},s)=L(j_{10}^{2,3},s) = L(j_{10}^{4,1},s)=L(j_{10}^{1,4},s).$$ 
The conductor $\bbf_{10}^{1,b}$ ($b=2$, $4$, $6$) divides $(2)^2(1-\z_5)^2$ and is as listed below. 
The CM type is $\s_1+\s_7$ for $b=2$, $4$ and $\s_1+\s_3$ for $b=6$. 
The group $(\sO_5/(2)(1-\z_5)^2)^\times$ is generated by $1+\z_5 \in \sO_5^\times$ and $1+2 \z_5$, which generates a prime ideal of degree one above $11$. 
One computes:
\begin{alignat*}{2}
& \bbf_{10}^{1,2} = (2)(1-\z_5)^2, \quad & & \vphi_{10}^{1,2}(\a)=-\z_5^4,  \\
& \bbf_{10}^{1,4} = (2)(1-\z_5), & & \vphi_{10}^{1,4}(\a)=-1, \\
& \bbf_{10}^{1,6} = (2)(1-\z_5)^2, &  &\vphi_{10}^{1,6}(\a)=-\z_5^2. \\
\end{alignat*}
By numerical computations, we obtain
\begin{align}
& L^*(j_{10}^{1,2},0) \approx  4R_{10}^{1,2}, \\
&L^*(j_{10}^{1,4},0) \approx  R_{10}^{1,4}, \\
&L^*(j_{10}^{1,6},0) \approx 4R_{10}^{1,6}, \\
&L^*(j_{10}^{2,5},0) \approx 4R_{10}^{2,5},
\end{align}
with $100$ digits precision. 
Since $R_{10}^{a,b} \neq 0$ in these cases, we obtain: 
\begin{thm}
The regulator map $r_\sD \ot_\Q \R$ is surjective for $X_{10}^{[1,2]}$, $X_{10}^{[1,4]}$, $X_{10}^{[1,6]}$ and $X_{10}^{[2,5]}$. 
\end{thm}

For $N=12$, by Proposition \ref{decomposition}, 
we are reduced to the cases where $a=1$, and $(a,b)=(2,3)$, $(3,4)$. 
When $a=1$, we can only treat the cases where $b$ is even. 
The cases $(a,b)=(1,4)$, $(1,8)$ and $(3,4)$ are excluded since we have 
 $r_\sD(e_{N,\a}^{[a,b]})=0$ by Proposition \ref{even N}. 
The case $(a,b)=(1,10)$ is excluded by Corollary \ref{restriction even}. 
For the remaining cases, the regulator determinants are given as follows; 
we use $e_{12}^{[a,b]}$, $e_{12,\a}^{[a,b]}$ for the first two and $e_{12}^{[a,b]}$, $e_{12,\b}^{[a,b]}$ for the last one: 
\begin{align*}
& R_{12}^{1,2} := \frac{D_{12}}{(2\pi)^2 12^2}   \left| \det
\begin{pmatrix}
F_{12}^{1,2} & 4 \frac{\sin\frac{9\pi}{12}}{\sin\frac{\pi}{12}} F_{12}^{9,2}\\
F_{12}^{7,2} & -4 \frac{\sin\frac{3\pi}{12}}{\sin\frac{7\pi}{12}} F_{12}^{3,2}
\end{pmatrix} \right|, \\
& R_{12}^{1,6} :=  \frac{D_{12}}{(2\pi)^2 12^2}  \left| \det
\begin{pmatrix}
F_{12}^{1,6} & -4 \frac{\sin\frac{5\pi}{12}}{\sin\frac{\pi}{12}} F_{12}^{5,6}\\
F_{12}^{5,6} & -4 \frac{\sin\frac{\pi}{12}}{\sin\frac{5\pi}{12}} F_{12}^{1,6}
\end{pmatrix} \right|, \\
& R_{12}^{2,3} :=  \frac{D_{12}}{(2\pi)^2 12^2} \left| \det
\begin{pmatrix}
F_{12}^{2,3} & -4 \frac{\sin\frac{7\pi}{12}}{\sin\frac{3\pi}{12}}  F_{12}^{2,7}\\
F_{12}^{2,9} &  4 \frac{\sin\frac{ \pi}{12}}{\sin\frac{9\pi}{12}} F_{12}^{2,1}
\end{pmatrix} \right|. 
\end{align*}
For the $L$-functions, we have by Proposition \ref{L symmetry}
$$L(j_{12}^{2,3},s)=L(j_{12}^{1,2},s).$$ 
The conductors are
$$\bbf_{12}^{1,2}=\bbf_{12}^{1,6}=(1-\z_4)^3(1-\z_3).$$
The CM types of $j_{12}^{1,2}$, $j_{12}^{1,6}$ are $\s_1+\s_7$, $\s_1+\s_5$, respectively. 
The group $(\sO_{12}/(1-\z_4)^3(1-\z_3))^\times$ is generated by 
$\z_{12}$, $1+(1-\z_4)\z_3 \in \sO_{12}^\times$, and $2-\z_{12}$, which generates a prime ideal of degree one above $13$. 
One computes
$$\vphi_{12}^{1,2}(2-\z_{12})=-\z_{12}, \quad \vphi_{12}^{1,6}(2-\z_{12})=-1.$$
By numerical computations, we obtain 
\begin{align}
& L^*(j_{12}^{1,2},0) \approx 3 R_{12}^{1,2},\\
& L^*(j_{12}^{1,6},0) \approx  -6 R_{12}^{1,6},\\
& L^*(j_{12}^{2,3},0) \approx 6 R_{12}^{2,3}, 
\end{align}
with $100$ digits precision.  
Since $R_{12}^{a,b} \neq 0$ in these cases, we obtain: 

\begin{thm} 
The regulator map $r_\sD \ot_\Q \R$ is surjective for 
$X_{12}^{[1,2]}$, $X_{12}^{[1,6]}$ and $X_{12}^{[2,3]}$.  
\end{thm}

\begin{rmk}
As in Remark \ref{hg relation 1}, the above results suggest quadratic relations
$$R_{10}^{1,4}=4 R_{10}^{2,5}, \quad R_{12}^{1,2}=2 R_{12}^{2,3}$$ 
among hypergeometric values. 
\end{rmk}

\subsection{$g=3$}

As we remarked after Proposition \ref{even N}, our three elements can possibly be linearly independent for $N=7$ and $9$. 

For $N=7$, by Proposition \ref{decomposition}, we are reduced to the cases $(a,b)=(1,1)$ and $(1,2)$. 
By Corollary \ref{restriction}, we can only treat $(a,b)=(1,2)$. 
Then, by Theorem \ref{regulator} and Propositions \ref{Q-structure} and \ref{odd N}, 
the regulator determinant with respect to $e_7^{[1,2]}$, $e_{7,\a}^{[1,2]}$, $e_{7,\b}^{[1,2]}$ is given by
\begin{equation*}
R_7^{1,2}:=\frac{D_7}{(2 \pi)^3 7^3} \left| \det \begin{pmatrix}
F_7^{1,2} & \frac{\sin\frac{4\pi}{7}}{\sin\frac{\pi}{7}} F_7^{4,2} & -\frac{\sin\frac{4\pi}{7}}{\sin\frac{2\pi}{7}} F_7^{1,4}  \\
F_7^{2,4} & \frac{\sin\frac{\pi}{7}}{\sin\frac{2 \pi}{7}} F_7^{1,4} & \frac{\sin\frac{\pi}{7}}{\sin\frac{4\pi}{7}} F_7^{2,1} \\
F_7^{4,1} & -\frac{\sin\frac{2\pi}{7}}{\sin\frac{4 \pi}{7}} F_7^{2,1} & \frac{\sin\frac{2\pi}{7}}{\sin\frac{\pi}{7}} F_7^{4,2} 
\end{pmatrix}\right|.
\end{equation*}
Using the expression
$$R_7^{1,2}=\frac{1}{56(2\pi)^3}(s^3+t^3+u^3-3stu), \quad s:=\frac{F_7^{1,2}}{\sin\frac{4\pi}{7}}, \ t:=-\frac{F_7^{2,4}}{\sin\frac{\pi}{7}}, \quad
u:=\frac{F_7^{4,1}}{\sin\frac{2\pi}{7}},$$
$R_7^{1,2} \neq 0$ is proved without numerical computation (\cite{otsubo-1}, Proposition 4.36).  
In particular, $r_\sD \ot_\Q \R$ is surjective for $X_7^{[1,2]}$. 
For the $L$-function, we have $$\mathbf{f}_7^{1,2}=(1-\z_7)$$ and $(\sO_7/(1-\z_7))^\times = \F_7^\times$ is generated by $3$. Since $3$ is inert in $K_7$,  we have 
$$\vphi_7^{1,2}(3)=-1.$$ 
By numerical computations, we obtain
\begin{equation}
L^*(j_7^{1,2},0) \approx \frac{1}{49} R_7^{1,2}
\end{equation}
with $100$ digits precision. 

Finally for $N=9$, by Proposition \ref{decomposition}, we are reduced to the cases $(a,b)=(1,1)$ and $(1,2)$, 
but we can only treat the latter case as above. 
The regulator determinant with respect to $e_9^{[1,2]}$, $e_{9,\a}^{[1,2]}$, $e_{9,\b}^{[1,2]}$ is given by
\begin{equation*}
R_9^{1,2} =  \frac{D_9}{(2\pi)^3 9^3}\left|\det\begin{pmatrix}
F_9^{1,2} & \frac{\sin\frac{6\pi}{9}}{\sin\frac{\pi}{9}} F_9^{6,2} &- \frac{\sin\frac{6 \pi}{9}}{\sin\frac{2 \pi}{9}} F_9^{1,6} \\
F_9^{2,4} & \frac{\sin\frac{3 \pi}{9}}{\sin\frac{2 \pi}{9}} F_9^{3,4} & \frac{\sin\frac{3 \pi}{9}}{\sin\frac{4 \pi}{9}} F_9^{2,3}\\
F_9^{5,1} & - \frac{\sin\frac{3\pi}{9}}{\sin\frac{5\pi}{9}} F_9^{3,1}& - \frac{\sin\frac{3 \pi}{9}}{\sin\frac{\pi}{9}} F_9^{5,3}
\end{pmatrix}\right|.
\end{equation*}
For the $L$-function, we have
$$\bbf_9^{1,2}=(1-\z_9)^4$$ and the CM type is $\s_1+\s_2+\s_5$.  
The group $(\sO/\bbf_9^{1,2})^\times$ is generated by $-\z_9 \in \sO_9^\times$ and $1+\z_9-\z_9^2$. 
The latter generates a prime ideal of degree one above $19$ and one computes
$$\vphi_9^{1,2}(1+\z_9-\z_9^2)=\z_9^8.$$ 
By numerical computations, we obtain
\begin{equation}
L^*(j_9^{1,2},0) \approx \frac{1}{3} R_9^{1,2}
\end{equation}
with $100$ digits precision. 
Since $R_9^{1,2} \neq 0$, we obtain: 
\begin{thm}
The regulator map $r_\sD \ot_\Q \R$ is surjective for $X_9^{[1,2]}$.  
\end{thm}

\section{Integral version}

We introduce $\Z$-structures on the motivic and the Deligne cohomologies of Fermat motives and renormalize the results in the preceding section.

\subsection{Integral structure of Deligne cohomology}

Here we determine the $\Z$-structure of the Deligne cohomology. 
We write 
$$H_1(X_N^{[a,b]}(\C),\Z) = p_{N *}^{[a,b]}H_1(X_N(\C),\Z)$$ 
and put
$$T_N^{[a,b]}=H_1(X_N^{[a,b]}(\C),\Z)^{F_\infty=-1}.$$
Then, the {\em $\Z$-structure} of Deligne cohomology is defined by
$$p_N^{[a,b]*}H^1(X_N(\C),\Z(1))^+ = \Hom(T_N^{[a,b]},\Z(1)),$$
where $\Z(1)=2 \pi i \Z$. 
We know (\cite{g-r-2}, Appendix) that  $H_1(X_N(\C),\Z)$ is a cyclic $\Z[G_N]$-module with a generator $\k$ such that 
\begin{align}\label{Fkappa}
& F_\infty \k = g^{-1,-1}_*\k,\\ \label{Okappa}
& \int_\k \wt\o_N^{a,b}=(1-\z_N^a)(1-\z_N^b).
\end{align}
Assume that $(a,b) \in I_N^\prim$. For any $g \in G_N$, we have 
$$p_{N *}^{[a,b]} \circ g_* (\k)= \sum_{h \in (\Z/N)^\times} \theta_N^{a,b}(g)^h p_{N *}^{ha,hb} (\k).$$
Note that this depends only on $\theta_N^{a,b}(g)$; fix $g\in G_N$ such that $\theta_N^{a,b}(g)=\z_N$. 
For each $n \in \Z/N$, define  
$$\k_n = p_{N *}^{[a,b]} \circ g^n_* (\k) =\sum_{h \in (\Z/N)^\times} \z_N^{hn} p_{N*}^{ha,hb} (\k) 
\ \in H_1(X_N^{[a,b]}(\C),\Z). $$

\begin{ppn}\label{Tbasis} 
Let the notations be as above. Then we have: 
\begin{enumerate}
\item $\{\k_n \mid n\in \Z/N\}$ generates $H_1(X_N^{[a,b]}(\C),\Z)$. 
\item If $\z_N$ is a root of $\sum_{n=0}^{N-1} a_n t^n \in \Z[t]$, then  
we have  $\sum_{i=0}^{N-1} a_n \k_n=0$. 
\item $F_\infty \k_n = \k_{c-n}$. 
\end{enumerate}
\end{ppn}

\begin{proof}(i) and (ii) are easy. 
By \eqref{Fkappa}, we have
$$F_\infty \circ g^n_*(\k)
=g^{-n}_*\circ F_\infty(\k)
=g^{-n}_*\circ g^{-1,-1}_*(\k).$$
Since $F_\infty$ commutes with $p_N^{[a,b]}$, we have
$$F_\infty \k_n
= p_{N*}^{[a,b]} \circ g^{-n}_*\circ g^{-1,-1}_*(\k)=\sum_{h\in (\Z/N)^\times} \z_N^{-h(n+a+b)}p_{N*}^{ha,hb}(\k)=\k_{c-n},$$
hence (iii) is proved. 
\end{proof}

This proposition enables us to find a $\Z$-basis of $T_N^{[a,b]}$ in each case. 
We only give one example;  
other cases are similarly determined (see Table \ref{table} below). 

\begin{ex}
Let $N=10$, $(a,b)=(1,2)$. Then, by (i) and (ii), 
$\{\k_n \mid n=0, \dots, 3\}$ is a basis of $H_1(X_{10}^{[1,2]}(\C),\Z)$. By (iii) and (ii), we have 
\begin{align*}
& F_\infty \k_0=\k_7=-\k_2, \quad F_\infty\k_1=\k_6=-\k_1, \\ 
&F_\infty\k_2=\k_5=-\k_0, \quad F_\infty \k_3=\k_4=\k_3-\k_2+\k_1-\k_0. 
\end{align*}
Therefore, a basis of $T_{10}^{[1,2]}$ is given by 
$\k_0-F_\infty \k_0=\k_0+\k_2$ and $\k_1$. 
\end{ex}

We put 
$$\k_n^-:=\k_n-F_\infty\k_n=\k_n-\k_{c-n}\ \in T_N^{[a,b]}.$$
Note that, if $\k_n \in T_N^{[a,b]}$ itself, then $\k_n=\k_n^-/2$.
The periods along $\k_n$ and $\k_n^-$ are given as follows. 

\begin{ppn}\label{Tperiod}
For any $n \in \Z/N$ and $h \in(\Z/N)^\times$, we have
\begin{align*}
& \int_{\k_n} \wt\o_N^{ha,hb} = \z_N^{hn}(1-\z_N^{ha})(1-\z_N^{hb}),\\
& \int_{\k_n^-}(\wt\o_N^{ha,hb}-\wt\o_N^{-ha,-hb}) 
=4i \Im(\z_N^{hn}(1-\z_N^{ha})(1-\z_N^{hb})).
\end{align*}
\end{ppn}

\begin{proof}
Since
$$
\int_{\k_n}\wt\o_N^{ha,hb} 
= \sum_{h' \in(\Z/N)^\times} \z_N^{h'n} \int_{\k} (p_N^{h'a,h'b})^*\wt\o_N^{ha,hb} \\
= \z_N^{hn}\int_\k \wt\o_N^{ha,hb},$$
the first formula follows from \eqref{Okappa}.  
Since 
$$\int_{F_\infty \k_n}\wt\o_N^{ha,hb} = \int_{\k_n} F_\infty\wt\o_N^{ha,hb}
= \int_{\k_n} \wt\o_N^{-ha,-hb},$$
the second formula follows from the first.
\end{proof}

\begin{rmk}\label{kappa gamma}
The relation between the $\Q$-structure given in Proposition \ref{Q-structure} and the $\Z$-structure given above is as follows. 
There exists an element $\g \in H_1(X_N(\C),\Q)$ such that 
$F_\infty \g=\g$, $\int_\g \wt\o_N^{a,b}=1$ for all $(a,b) \in I_N$, and 
$\k=((1-g^{1,0})(1-g^{0,1}))_*(\g)$ (see \cite{otsubo-1}, \S4.6). By the same procedure as above starting with $\g$ instead of $\k$, we obtain 
a basis $\{\g_n \mid n=1,\dots, 2g\}$ of $H_1(X_N^{[a,b]}(\C),\Q)$ such that $\int_{\g_n}\wt\o_N^{ha,hb}=\z_N^{hn}$ for any $h \in (\Z/N)^\times$. 
If $\{\g_n^\vee \mid n=1,\dots, 2g\}$ denotes the dual basis of $H^1(X_N^{[a,b]}(\C),\Q)$, then we have 
$\l_n=2\pi i (\g_n^\vee-F_\infty \g_n^\vee)$. 
\end{rmk}

\subsection{Regulators}

As in \cite{d-j-z}, for a curve $X$ over $\Q$, we define the {\em $\Z$-structure} 
of $H^2_\sM(X,\Q(2))_\Z$ to be
$$\Im\bigl(\Ker(T) \ra H^2_\sM(X,\Q(2))\bigr) \cap H^2_\sM(X,\Q(2))_\Z.$$
For Fermat motives, we define the $\Z$-structure of $H_\sM^2(X_N^{[a,b]},\Q(2))_\Z$ 
to be the image of that of $H^2_\sM(X_N,\Q(2))_\Z$ under $p_N^{[a,b]*}$. 
Then, our elements 
$e_N^{[a,b]}$, $e_{N,\a}^{[a,b]}$, $e_{N,\b}^{[a,b]}$ belong to the $\Z$-structure by definition. 

For each $(a,b) \in I_N^\prim$, choose a $\Z$-basis $\{t_n \mid n=1,\dots, g\}$ of $T_N^{[a,b]}$ and put
$$
D_N^{a,b}= \left|\det \left( \int_{t_n}(\wt\o_N^{ha,hb}-\wt\o_N^{-ha,-hb})\right)_{h\in H_N^{a,b},n=1,\dots,g}\right|,  
$$
which is independent of the choice of a basis. 
This can be computed using Propositions \ref{Tbasis} and \ref{Tperiod}.
Then, for each case of \S4 where we have the surjectivity of $r_\sD \ot_\Q \R$, 
the regulator determinant of our $g$ elements with respect to the $\Z$-structure 
of Deligne cohomology is given by
$$\wt R_N^{a,b} := \frac{D_N^{a,b}}{D_N}R_N^{a,b}.$$
In Table \ref{table} below, 
we summarize a basis of $T_N^{[a,b]}$, the period determinant $D_N^{a,b}$ 
and the ratio $\wt R_N^{a,b}/L^*(j_N^{a,b},0) \in \Q^\times$ (with $100$ digits precision) of the regulator determinant to the $L$-value.

\begin{table}[htdp]
\caption{}
\begin{center}
\begin{tabular}{c c c c c cc c}
\hline
\ $g$ \ &\  $N$ \ &\  $(a,b) \ $ & $H_N^{a,b}$ & $T_{N}^{[a,b]}$ & $D_N^{a,b}$ 
& \ $\wt R_N^{a,b}/L^*(j_N^{a,b},0)$\ \\
\hline
$1$ & $3$ & $(1,1)$ & $\{1\}$ & $\angle{\k_0^-}$ & $2 \cdot 3 \sqrt{3}$ & $2\cdot 3^2$\\
& $4$ & $(1,1)$  & $\{1\}$ & $\angle{\k_0^-/2}$ & $2^2$ & $2^2$\\
&        & $(1,2)$  & $\{1\}$ & $\angle{\k_0^-}$ & $2^3$ & $2^3$\\
& $6$ & $(1,1)$  & $\{1\}$ & $\angle{\k_0^-}$ & $2\sqrt{3}$ & $-1$\\
&  & $(1,2)$  & $\{1\}$ & $\angle{\k_1^-/2}$ & $2\sqrt{3}$ & $2\cdot 3$\\
&  & $(1,3)$  & $\{1\}$ & $\angle{\k_0^-}$ & $2\sqrt{3}$ & $1$\\
&  & $(1,4)$  & $\{1\}$ & $\angle{\k_0^-}$ & $2\sqrt{3}$ & $1$\\
&  & $(2,3)$  & $\{1\}$ & $\angle{\k_0^-}$ & $2^2\sqrt{3}$ & $2$\\
\hline
$2$ & $5$ & $(1,1)$  & $\{1, 2\}$ & $\angle{\k_0^-, \k_1^-}$ & $2^2\cdot 5^2$ & $2^2\cdot 5^2$\\
& $10$ & $(1,2)$ & $\{1, 3\}$ & $\angle{\k_0^-, \k_1^-/2}$ & $2\cdot 5$ & $1/2$  \\
& & $(1,4)$  & $\{1, 3\}$ & $\angle{\k_0^-/2, \k_1^-}$ & $2\cdot 5$ & $2$ \\
& & $(1,6)$  & $\{1, 7\}$ & $\angle{\k_0^-, \k_1^-}$ & $2^2\cdot 5$ & $1$\\
& & $(2,5)$  & $\{1, 7\}$ & $\angle{\k_0^-, \k_1^-}$ & $2^4\cdot 5$ & $2^2$\\
& $12$ & $(1,2)$ & $\{1, 7\}$ &  $\angle{\k_0^-, \k_1^-}$ & $2^3 \sqrt{3}$ & $2^2/3$\\
& & $(1,6)$  & $\{1, 5\}$ & $\angle{\k_0^-, \k_1^-}$ & $2^5\sqrt{3}$ & $-2^3/3$\\
& & $(2,3)$  & $\{1, 7\}$ & $\angle{\k_0^-, \k_2^-}$& $2^4\sqrt{3}$ & $2^2/3$\\
\hline
$3$ & $7$ & $(1,2)$  & $\{1, 2, 4\}$ & $\angle{\k_0^-, \k_1^-, \k_5^-}$ & $2^3 \cdot 7^2 \sqrt{7}$ & $2^3\cdot 7^3$\\
& $9$ & $(1,2)$ & \ $\{1, 2, 5\}$ \  & \ $\angle{\k_0^-, \k_1^-, \k_2^-}$ \ &  \ $2^3\cdot 3^3 \sqrt{3}$ \  & $2^3 \cdot 3^2$\\ 
\hline
\end{tabular}
\end{center}
\label{table}
\end{table}

\begin{rmk}\label{p-reg}
Bloch-Kato's Tamagawa number conjecture \cite{b-k} predicts the rational factor in terms of 
the $p$-adic regulators for all prime numbers $p$. Let  
$$V_p=H^1_{\mathrm{\acute{e}t}}(X_N^{[a,b]} \ot_\Q \ol \Q, \Q_p(2))$$
be the $p$-adic \'etale cohomology group, on which the absolute Galois group $G_\Q:=G(\ol \Q/\Q)$ acts continuously. 
We have the $p$-adic regulator map to the Galois cohomology 
$$r_p \colon H_\sM^2(X_N^{[a,b]},\Q(2))_\Z \ra H^1(G_\Q,V_p)$$
induced from the Chern class map. 
The conjecture states firstly that $r_p \ot_\Q \Q_p$ is isomorphic onto the {\em Selmer group} $H^1_f(G_\Q,V_p) \subset H^1(G_\Q,V_p)$. 
Secondly, the $p$-part of the rational factor in the Beilinson conjecture, coming from the indeterminacy of a $\Q$-basis of the motivic cohomology, is conjecturally detected by the ``index" of the $\Z_p$-module generated by the image of the basis under $r_p$, with respect to a canonical $\Z_p$-lattice $H^1_f(G_\Q,T_p) \subset H^1_f(G_\Q,V_p)$. 
In view of the conjecture, our results (see Table \ref{table}) suggest, as in the case of CM elliptic curves (\cite{b-k}, \S7), 
that, for $p \nmid 2N$, the images of our motivic elements generate the $\Z_p$-module $H^1_f(G_\Q,T_p)$. 
\end{rmk}

\bigskip

\begin{center}{\sc Acknowledgements}\end{center}
\smallskip

This article was written while the author was visiting University of Toronto.  
I would like to thank Kumar Murty and the department for their hospitality. 
I would like to thank Rob de Jeu, Christophe Soul\'e and Takuya Yamauchi for stimulating discussions. 
Finally, I would like to thank the contributors to Magma.  


\end{document}